\documentclass{bmcart}

\usepackage[utf8]{inputenc} 
\usepackage{color}
\usepackage{amsthm}
\usepackage{amsmath}
\usepackage{amssymb}
\usepackage{graphicx}
\usepackage{esint}

\def\includegraphics{}

\startlocaldefs

\providecommand{\tabularnewline}{\\}
\theoremstyle{plain}
\newtheorem{thm}{\protect\theoremname}
  \theoremstyle{plain}
  \newtheorem{lem}[thm]{\protect\lemmaname}
  \theoremstyle{plain}
  \newtheorem{prop}[thm]{\protect\propositionname}
\providecommand{\lemmaname}{Lemma}
  \providecommand{\propositionname}{Proposition}
\providecommand{\theoremname}{Theorem}

\endlocaldefs

\begin{document}

\begin{frontmatter}

\begin{fmbox}
\dochead{Research}


\title{Integration over curves and surfaces defined by the closest point
mapping}


\author[
   addressref={ud},                   
   corref={ud},                       
   email={ckublik1@udayton.edu}   
]{\inits{CK}\fnm{Catherine } \snm{Kublik}}
\author[
   addressref={ut,kth},
   email={ytsai@ices.utexas.edu}
]{\inits{RT}\fnm{Richard} \snm{Tsai}}


\address[id=ud]{
  \orgname{Department of Mathematics, University of Dayton}, 
  \street{300 College Park},                     %
  \city{Dayton, OH},                              
  \cny{USA}                                    
}
\address[id=ut]{%
  \orgname{Department of Mathematics and Institute for Computational Engineering and Sciences, University of Texas at Austin},
  \street{2515 Speedway},
  \postcode{78712}
  \city{Austin},
  \cny{USA}
  }
  
  \address[id=kth]{%
  \orgname{Department of Mathematics, KTH Royal Institute of Technology},
  \street{SE-100 44},
  \city{Stockholm},
  \cny{Sweden}
}



\end{fmbox}


\begin{abstractbox}

\begin{abstract} 
We propose a new formulation for integrating
over smooth curves and surfaces  that are described
by their closest point mappings. 
Our method is designed for curves and surfaces that are not defined by any explicit parameterization
and  is intended to be used in combination with level set techniques. 
However, contrary to the common practice
with level set methods, the \emph{volume integrals} derived from our
formulation coincide exactly with the surface or line integrals that
one wishes to compute. We study various aspects of this formulation
and provide a geometric interpretation of this formulation in terms
of the singular values of the Jacobian matrix of the closest point
mapping. Additionally, we extend the formulation - initially derived
to integrate over manifolds of codimension one - to include integration
along curves in three dimensions. Some numerical examples using very simple discretizations
are presented to demonstrate the efficacy of the formulation.

\end{abstract}




\end{abstractbox}
%

\end{frontmatter}




\section{Introduction}

This paper provides simple formulations for integrating over manifolds of codimensions one, or two in $\mathbb{R}^3$, when the manifolds are described by functions that map points in $\mathbb{R}^n$ ($n=2,3$) to their closest points on curves or surfaces using the Euclidean distance.
The idea for the present work originated in 
\cite{kublik_tanushev_tsai13} where the authors proposed a formulation for computing
integrals of the form 
\begin{equation}
\int_{\partial\Omega}v(\mathbf{x}(s))ds,\label{intoverboundary}
\end{equation}
in the level set framework, namely when the domain $\Omega$ is represented
implicitly by the signed distance function to its boundary $\partial\Omega$. 
Typically in a level set method \cite{osher_fedkiw02,osher_sethian88, sethian99}, 
to evaluate an integral of the form
of \eqref{intoverboundary} where $\partial\Omega$ is the zero level
set of a continuous function $\varphi$, it is necessary to extend
the function $v$ defined on the boundary $\partial\Omega$ to a neighborhood in $\mathbb{R}^{n}$.
The extension of $v$, denoted $\tilde{v}$, is typically a constant
extension of $v$. The integral is then \emph{approximated} by an
integral involving a regularized Dirac-$\delta$ function concentrated
on $\partial\Omega$, namely 
\[
\int_{\partial\Omega}v(\mathbf{x}(s))ds\approx\int_{\mathbb{R}^{n}}\tilde{v}(\mathbf{x})\delta_{\epsilon}(\varphi(\mathbf{x}))|\nabla\varphi(\mathbf{x})|d\mathbf{x}.
\]
Various numerical approximations of this delta function have been
proposed, see e.g. \cite{engquist_tornberg_tsai05,dolbow_harari09,smereka06,towers07,zahedi_tornberg10}.

In \cite{kublik_tanushev_tsai13}, with the choice of $\varphi=d_{\partial\Omega}$ 
being a signed distance function to $\partial\Omega$,  the integral \eqref{intoverboundary} is expressed as an average of integrals over nearby level sets of
$d_{\partial\Omega}$, where these nearby level sets continuously sweep a thin
tubular neighborhood around the boundary $\partial\Omega$ of radius
$\epsilon$. Consequently, \eqref{intoverboundary} is \emph{equivalent} to the volume integral shown on the right hand side below:
\begin{equation}
\int_{\partial\Omega}v(\mathbf{x}(s))ds=\int_{\mathbb{R}^{n}}v(\mathbf{x}^{*})J(\mathbf{x};d_{\partial\Omega})\delta_{\epsilon}(d_{\partial\Omega}(\mathbf{x}))d\mathbf{x},\label{intformulation}
\end{equation}
where $\delta_{\epsilon}$ is an averaging kernel, $\mathbf{x}^{*}$
is the
closest point  on $\partial\Omega$ to $\mathbf{x}$  and 
$J(\mathbf{x};d_{\partial\Omega})$  accounts 
for the change in curvature between the nearby level sets and the zero level set.

Now suppose that $\partial\Omega$ is a smooth hypersurface in $\mathbb{R}^3$ and assume that $\mathbf{x}$ is sufficiently close to $\Omega$ so that the closest point mapping
\[
\mathbf{x}^*=P_{\partial\Omega}(\mathbf{x})=\mathrm{argmin}_{y\in\partial\Omega} |\mathbf{x}-\mathbf{y}|
\]
 is continuously differentiable. 
 Then the restriction of $P_{\partial\Omega}$ to  $\partial\Omega_\eta$ is a diffeormorphism between $\partial\Omega_\eta$ and $\partial\Omega$, where $\partial \Omega_{\eta} \mathrel{\mathop:}=  \left \{ \mathbf{x} : d_{\partial \Omega} (\mathbf{x}) = \eta \right \}$.
As a result, it is possible to write integrals over $\partial\Omega$ using points on $\partial\Omega_\eta$ as:
\[
\int_{\partial\Omega} v(\mathbf{x}) dS = \int_{\partial\Omega_\eta} v(\mathbf{x}^*)J(\mathbf{x};\eta)dS,
\]
where $J(\mathbf{x},\eta)$ comes from the change of variable defined by $P_{\partial\Omega}$ restricted on $\partial\Omega_\eta$.
Averaging the above integrals respectively with a kernel, $\delta_\epsilon$, compactly supported in $[-\epsilon,\epsilon]$, we obtain
\[
\int_{\partial\Omega} v(\mathbf{x}) dS = \int_{-\epsilon}^{\epsilon} \delta_\epsilon(\eta) \int_{\partial\Omega_\eta} v(\mathbf{x}^*)J(\mathbf{x};\eta)dS~d\eta.
\]
Formula \eqref{intformulation} then follows from the coarea formula \cite{federer69} applied to the integral on the right hand side.

In the following section, we show that in three dimensions
the Jacobian $J$ in~\eqref{intformulation} is the product of the first two singular values, $\sigma_1$ and $\sigma_2$, of the Jacobian matrix of the closest point mapping $\frac{\partial P_{\partial\Omega}}{\partial \mathbf{x}}$; namely,
\begin{equation}
\int_{\partial\Omega}v(\mathbf{x}(s))ds=\int_{\mathbb{R}^{3}}v(P_{\partial\Omega }(\mathbf{x}))\delta_{\epsilon}(d_{\partial \Omega}(\mathbf{x}))\prod_{j=1}^{2} \sigma_j(\mathbf{x}) d\mathbf{x}. \label{our_formulation}
\end{equation}
To motivate the new approach using singular values, we consider  
Cartesian coordinate systems with the origin placed on points sufficiently close to the surface, 
and the $z$ direction normal to the surface. Thus the partial derivatives of the closest point mapping in the $z$ direction will yield zero and the partial derivatives in the other two directions naturally 
correspond to differentiation in the tangential directions. 
Thus we see that one of the singular values should be $0$ while the other two are related to the surface area element. 
We also derive a similar formula for integration along curves in three dimensions (codimension 2).
The advantages of this new formula include the ease for constructing higher order approximations 
of $J$ via e.g. simple differencing, even in neighborhoods of surface boundaries where curvatures become unbounded.

This paper is motivated by the recent success in the closest point methods and the Dynamic Surface Extension method \cite{steinhoff2000new}, for evolving interfaces and solving partial differential equations on 
surfaces \cite{macdonald_brandman_ruuth11, macdonald_ruuth08, macdonald_ruuth09, ruuth_merriman08}, by the need to process data sets that contain 
unstructured points sampled from some underlying surfaces, and targets applications where manifolds are not defined by patches of explicit parameterizations and may evolve drastically due to some coupled processes; 
see e.g. free boundary problems~\cite{Hou-Acta-Numerica-free-boundary}.
Our work provides a convenient way to formulate boundary integral methods in such applications without conversion to local parameterizations.
If the manifolds are defined by explicit parameterizations, it is natural and typically more accurate to use conventional methods 
such as Nystr\"{o}m methods using quadratures on the parameter space or Boundary Element Methods with weak formulations, see
e.g. \cite{Atkinson-book}. 
Additionally, for applications involving fluid-structure interactions, we mention the immersed boundary method which involves accurate discretizations of Dirac delta measures  \cite{mittal_iaccarino05,peskin02}.

Closest point mappings are easily computed in the context of level set methods \cite{osher_sethian88} since there exist fast algorithms for constructing distance functions from level set functions \cite{cheng_tsai08, russo_smereka00, sethian96, tsai_cheng_osher_zhao03,tsitsiklis95}. More precisely, 
\[
P_{\partial\Omega}(\mathbf{x})=\mathbf{x}-d_{\partial\Omega}(\mathbf{x})\nabla d_{\partial\Omega}(\mathbf{x}).
\]
Our previous work \cite{kublik_tanushev_tsai13} as well as this current paper provide a simple framework for constructing numerical schemes for boundary integral methods when the interface is described implicitly by a level set function, and is intended for use in such context.

Finally, closest point mappings can also be computed easily from dense and unorganized point sets that are acquired directly from an imaging device (e.g. LIDAR). This paper lays the foundation of a numerical scheme for computing integrals over surfaces sampled by unstructured point clouds. 

\section{Integration using the closest point mapping\label{sec:Closest-point-mapping}}

In this section, we relate the Jacobian $J$ in \eqref{intformulation}
to the singular values of the Jacobian matrix
of the closest point mapping from $\mathbb{R}^2$ or $\mathbb{R}^3$ to $\Gamma$, where $\Gamma$ denotes the curves or surfaces on which integrals are defined.
We assume that in three dimensions, if $\Gamma$
is not closed, it has smooth boundaries. For clarity of the exposition in the rest of the paper, we will now denote the distance function simply by $d$.

\subsection{Codimension 1}

We consider a $C^{2}$ compact curve or surface $\Gamma$ that can
either be closed or not. If $\Gamma$ is closed, then it is the boundary
of a domain $\Omega$ so that $\Gamma$ can be denoted $\partial\Omega$.
If $\Gamma$ is not closed, we assume that it has smooth boundaries.
We define $d:\mathbb{R}^{n}\mapsto\mathbb{R}\cup\{0\}$ to be the
distance function to $\Gamma$ and $P_{\Gamma}$ to be the closest
point mapping $P_{\Gamma}:\mathbb{R}^{n}\mapsto\Gamma$ (for $n=2,3$)
defined as

\begin{equation}
|P_{\Gamma}(\mathbf{x})-\mathbf{x}|=\min_{\mathbf{y}\in\Gamma}|\mathbf{y}-\mathbf{x}|.\label{eq:cpm}
\end{equation}

We let $d_{0}$ be the distance function to $\Gamma$ if it is open
and $d_{s}$ be the signed distance function to $\Gamma=\partial\Omega$
if it is closed. The signed distance function is defined as

\[
d_{s}(\mathbf{x})\mathrel{\mathop:}=\begin{cases}
\inf_{\mathbf{y}\in\Omega^{c}}|\mathbf{x}-\mathbf{y}| & \mbox{if }\mathbf{x}\in\Omega,\\
-\inf_{\mathbf{y}\in\Omega}|\mathbf{x}-\mathbf{y}| & \mbox{if }\mathbf{x}\in\bar{\Omega}^{c}.
\end{cases}
\]
Then we define $d$ as follows: 
\begin{equation}
d(\mathbf{x})\mathrel{\mathop:}=\begin{cases}
d_{0}(\mathbf{x}) & \mbox{if }\Gamma\mbox{ is open,}\\
d_{s}(\mathbf{x}) & \mbox{if }\Gamma\mbox{ is closed.}
\end{cases}\label{eq:distancefn}
\end{equation}

The following lemma provides a concise expression of the Gaussian curvature in terms of the distance function. This is probably a known result but we include its proof to preserve the completeness of the paper.

\begin{lem}
\label{gaussiancurv_lemma} Let  $d$ be the distance function
to  $\Gamma$ defined in~\eqref{eq:distancefn}. For $|\eta|$  sufficiently close to $0$, the Gaussian curvature
at a point on the $\eta$ level set $\Gamma_{\eta} \mathrel{\mathop:}=  \left\{ \xi:d(\xi)=\eta\right\} $ can be expressed as
\begin{equation}
G_{\eta}=d_{xx}d_{yy}+d_{xx}d_{zz}+d_{yy}d_{zz}-d_{xy}^{2}-d_{xz}^{2}-d_{yz}^{2}.\label{gaussiancurv_exp2}
\end{equation}
\end{lem}
\begin{proof}

Starting with the definition of the Gaussian curvature $G$ for a surface (see \cite{hicks65}), we can obtain an expression for the Gaussian curvature of its $\eta$-level set in terms of $d$ as
\begin{align}
G &=\langle\nabla d,adj(Hess(d))\nabla d\rangle\nonumber \\
 & =d_{x}^{2}(d_{yy}d_{zz}-d_{yz}^{2})+d_{y}^{2}(d_{xx}d_{zz}-d_{xz}^{2})+d_{z}^{2}(d_{xx}d_{yy}-d_{xy}^{2})\notag\\
 & \,\,+2[d_{x}d_{y}(d_{xz}d_{yz}-d_{xy}d_{zz})+~d_{y}d_{z}(d_{xy}d_{xz}-d_{yz}d_{xx})\notag\label{gaussian_lastline}\\
 & \,\,+d_{x}d_{z}(d_{xy}d_{yz}-d_{xz}d_{yy})].
\end{align}
We show that this expression is the same as \eqref{gaussiancurv_exp2}
by rearranging the terms above and using the fact that close to $\Gamma$ the distance function satisfies $|\nabla d| = 1$.
First we rearrange the terms in $G$: 
\begin{eqnarray*}
G & = & d_{x}^{2}d_{yy}d_{zz}+d_{y}^{2}d_{xx}d_{zz}+d_{z}^{2}d_{xx}d_{yy}-d_{x}^{2}d_{yz}^{2}-d_{y}^{2}d_{xz}^{2}-d_{z}^{2}d_{xy}^{2}\\
 &  & +2[d_{x}d_{y}(d_{xz}d_{yz}-d_{xy}d_{zz})+~d_{y}d_{z}(d_{xy}d_{xz}-d_{yz}d_{xx})+d_{x}d_{z}(d_{xy}d_{yz}-d_{xz}d_{yy})],
\end{eqnarray*}
and rewrite each of the first six terms in terms of $|\nabla d|^{2}$,
e.g. 
\[
d_{x}^{2}d_{yy}d_{zz}=\underbrace{|\nabla d|^{2}}_{=1}d_{yy}d_{zz}-d_{y}^{2}d_{yy}d_{zz}-d_{z}^{2}d_{yy}d_{zz}=d_{yy}d_{zz}-d_{y}^{2}d_{yy}d_{zz}-d_{z}^{2}d_{yy}d_{zz}.
\]
Thus we have 
\begin{eqnarray}
 &  & d_{x}^{2}d_{yy}d_{zz}+d_{y}^{2}d_{xx}d_{zz}+d_{z}^{2}d_{xx}d_{yy}-d_{x}^{2}d_{yz}^{2}-d_{y}^{2}d_{xz}^{2}-d_{z}^{2}d_{xy}^{2}\nonumber \\
 & = & \underbrace{d_{xx}d_{yy}+d_{xx}d_{zz}+d_{yy}d_{zz}-d_{xy}^{2}-d_{xz}^{2}-d_{yz}^{2}}_{=G_\eta}-d_{y}^{2}d_{yy}d_{zz}-d_{z}^{2}d_{yy}d_{zz}\nonumber \\
 &  & -d_{x}^{2}d_{xx}d_{zz}-d_{z}^{2}d_{xx}d_{zz}-d_{y}^{2}d_{xx}d_{yy}-d_{x}^{2}d_{xx}d_{yy}\label{eq:GaussianCurv}\\
 &  & +d_{y}^{2}d_{yz}^{2}+d_{z}^{2}d_{yz}^{2}+d_{x}^{2}d_{xz}^{2}+d_{z}^{2}d_{xz}^{2}+d_{x}^{2}d_{xy}^{2}+d_{y}^{2}d_{xy}^{2}\nonumber \\
\nonumber 
\end{eqnarray}
Using \eqref{eq:GaussianCurv} and rearranging the rest of the terms
in \eqref{gaussian_lastline} we obtain $G=G_{\eta}$. \end{proof}
\begin{prop}
\label{thm_jacobian_evalues} Consider a $C^{2}$ compact surface
$\Gamma\subset\mathbb{R}^{n}$ ($n=2,3$) of codimension~$1$ and
let $d$ be defined as in \eqref{eq:distancefn}. Define the closest
point projection map $P_{\Gamma}$ as in \eqref{eq:cpm} for $\mathbf{x}\in\mathbb{R}^{n}$.
For $|\eta|$ sufficiently close to zero, let $\Gamma_{\eta}$ be the $\eta$ level
set of $d$ 
\begin{equation}
\Gamma_{\eta}\mathrel{\mathop:}=\left\{ \mathbf{x}:d(\mathbf{x})=\eta\right\} . \label{gamma_eta}
\end{equation}
Define the Jacobian $J_{\eta}$ as 
\[
J_{\eta}\mathrel{\mathop:}=\left\{ \begin{array}{ll}
1+\eta\kappa_{\eta} & \mbox{ if }n=2,\\
1{\color{red}{\color{black}+}}2\eta H_{\eta}+\eta^{2}G_{\eta} & \mbox{ if }n=3,
\end{array}\right.
\]
where $\kappa_{\eta}$ is the signed curvature of $\Gamma_{\eta}$
in 2D, and $H_{\eta}$ and $G_{\eta}$ are its Mean curvature and
Gaussian curvature respectively in 3D.

Then if $P'_{\Gamma} $ is the Jacobian
matrix of $P_{\Gamma}$ we have 
\begin{equation}
J_{\eta}=\left\{ \begin{array}{ll}
\sigma_{1}, & n=2,\\
\sigma_{1}\sigma_{2}, & n=3,
\end{array}\right.\label{jacobian_singularvalues}
\end{equation}
where $\sigma_1, \sigma_2$ are the first two singular values of the
Jacobian matrix $P'_{\Gamma}$. \end{prop}
\begin{proof}
The distance function $d$ satisfies the property $d(\mathbf{x})=0$
for $\mathbf{x}\in\Gamma.$ Also, since $\Gamma$ is $C^{2}$, its distance function $d$ belongs to $C^{2}(\mathbb{R}^{n},\mathbb{R})$;
see e.g. \cite{delfour_zolesio01,federer59}. It follows that the
order of the mixed partial derivatives does not matter. In addition,
the normals to a smooth interface do not focus right away so that
the distance function is smooth in a tubular neighborhood $T$ around
$\Gamma$, and is linear with slope one along the normals. Therefore
we have 
\begin{equation}
|\nabla d|=1\mbox{ for all }\mathbf{x}\in T.\label{gradd_1}
\end{equation}
The third important fact is that the Laplacian of $d$ at a point
$\mathbf{x}$ gives (up to a constant related to the dimension) the
mean curvature of the isosurface of $d$ passing through $\mathbf{x}$,
namely 
\begin{equation}
\Delta d(\mathbf{x})=(1-n)H(\mathbf{x}),\label{laplacian_d_curvature}
\end{equation}
where $H(\mathbf{x})$ is the Mean curvature of the level set $\left\{ \mathbf{y}:d(\mathbf{y})=d(\mathbf{x})\right\} $.
Differentiating \eqref{gradd_1} with respect to each variable gives
the following equations in three dimensions: 
\begin{align}
d_{x}d_{xx}+d_{y}d_{xy}+d_{z}d_{xz}=0,\label{3d_eq1}\\
d_{x}d_{yx}+d_{y}d_{yy}+d_{z}d_{yz}=0,\label{3d_eq2}\\
d_{x}d_{zx}+d_{y}d_{zy}+d_{z}d_{zz}=0.\label{3d_eq3}
\end{align}
In particular the two dimensional case can be derived by assuming
that the distance function is constant in $z.$

\paragraph{Two dimensions. }

In that case the Jacobian matrix $P'_{\Gamma}$
of the closest point projection map is 
\[
P'_{\Gamma}=\left(\begin{array}{cc}
1-d_{x}^{2}-dd_{xx} & -(d_{y}d_{x}+dd_{yx})\\
-(d_{x}d_{y}+dd_{xy}) & 1-d_{y}^{2}-dd_{yy}
\end{array}\right).
\]
Since Schwartz' Theorem holds, we have $d_{xy}=d_{yx}$ making $P'_{\Gamma}$
a real symmetric matrix. It is therefore diagonalizable with eigenvalues
$0$ and $1-d\Delta d$. Indeed, we have 
\begin{eqnarray*}
P'_{\Gamma}\nabla d & = & \left(\begin{array}{c}
d_{x}(\underbrace{1-d_{x}^{2}-d_{y}^{2}}_{=0\mbox{ by }\eqref{gradd_1}\mbox{ in 2D}})-d(\underbrace{d_{x}d_{xx}+d_{y}d_{yx}}_{=0\mbox{ by }\eqref{3d_eq1}\mbox{ in 2D}})\\
\\
d_{y}(\underbrace{1-d_{x}^{2}-d_{y}^{2}}_{=0\mbox{ by }\eqref{gradd_1}\mbox{ in 2D}})-d(\underbrace{d_{y}d_{yy}+d_{x}d_{xy}}_{=0\mbox{ by }\eqref{3d_eq2}\mbox{ in 2D}})
\end{array}\right)=\mathbf{0},
\end{eqnarray*}
and for $\mathbf{v}=\left(\begin{array}{c}
-d_{y}\\
d_{x}
\end{array}\right)$, 
\begin{eqnarray*}
P'_{\Gamma}\mathbf{v} & = & \left(\begin{array}{c}
-d_{y}+d_{y}d_{x}^{2}+d_{y}dd_{xx}-d_{x}^{2}d_{y}-dd_{x}d_{xy}\\
d_{y}^{2}d_{x}+d_{y}dd_{xx}-d_{y}^{2}d_{x}-d_{x}dd_{yy}
\end{array}\right)\\
 & = & \left(\begin{array}{c}
-d_{y}\\
d_{x}
\end{array}\right)+d\left(\begin{array}{c}
d_{y}d_{xx}-d_{x}d_{xy}\\
d_{y}d_{xy}-d_{x}d_{yy}
\end{array}\right)\\
\\
 & = & v+d\left(\begin{array}{c}
-\Delta d(-d_{y})-\underbrace{(d_{y}d_{yy}+d_{x}d_{xy})}_{=0\mbox{ by }\eqref{3d_eq2}\mbox{ in 2D}}\\
\\
-\Delta d(dx)+\underbrace{d_{x}d_{xx}+d_{y}d_{xy}}_{=0\mbox{ by }\eqref{3d_eq1}\mbox{ in 2D}}
\end{array}\right)\\
 & = & (1-d\Delta d)\mathbf{v}.
\end{eqnarray*}
Since $||\mathbf{v}||=1$, $\mathbf{v}$ is an eigenvector corresponding
to the eigenvalue $\lambda=1-d\Delta d$. Thus, for $\mathbf{x}$
such that $d(\mathbf{x})=\eta$ we have that the eigenvalue $\lambda$ of $P'_{\Gamma}$
satisfies 
\[
\lambda=1-\eta\Delta d=1+\eta\kappa_{\eta}
\]
by \eqref{laplacian_d_curvature}. Since $1+\eta\kappa_{\eta}\geq0$,
it follows that $\lambda$ coincides with the singular value of $P'_{\Gamma}$
and hence
\[
\sigma_1=1+\eta\kappa_{\eta}.
\]

\paragraph{Three dimensions. }

Since for $|\eta|$ sufficiently close to $0$ the distance function is $C^2$,  the Jacobian matrix 
\[
P'_{\Gamma}=\left(\begin{array}{ccc}
1-d_{x}^{2}-dd_{xx} & -(d_{y}d_{x}+dd_{yx}) & -(d_{z}d_{x}+dd_{zx})\\
-(d_{x}d_{y}+dd_{xy}) & 1-d_{y}^{2}-dd_{yy} & -(d_{z}d_{y}+dd_{zy})\\
-(d_{x}d_{z}+dd_{xz}) & -(d_{y}d_{z}+dd_{yz}) & 1-d_{z}^{2}-dd_{zz}
\end{array}\right),
\]
is a real symmetric matrix which is diagonalizable with one zero eigenvalue and two other eigenvalues $\lambda_{1}$ and $\lambda_{2}$. Indeed
using \eqref{3d_eq1},\eqref{3d_eq2},\eqref{3d_eq3} and \eqref{gradd_1}
we can show that 
\[
P'_{\Gamma}\nabla d=\mathbf{0}.
\]
Now consider $\mathbf{x}$ such that $d(\mathbf{x)}=\eta$. Then,
the characteristic polynomial $\chi(\lambda)$ of $P'_{\Gamma}$
is 
\[
\chi(\lambda)=-\lambda\left(\lambda^{2}-(2-\eta\Delta d)\lambda-Q\right),
\]
where $Q=-G_\eta\eta^{2}+\eta\Delta d-1$ with $G_\eta$ defined in \eqref{gaussiancurv_exp2}.
Since the other two eigenvalues of $P'_{\Gamma}$
are the solutions of the quadratic equation $\lambda^{2}~-~(2~-~\eta\Delta d)\lambda~-~Q~=~0$,
it follows that 
\[
\lambda_{1}\lambda_{2}=-Q=1-\eta\Delta d+\eta^{2}G_\eta=1+2\eta H_{\eta}+\eta^{2}G_{\eta}.
\]
Since $1+2\eta H_{\eta}+\eta^{2}G_{\eta}\geq0$, it follows that
\[
\sigma_{1}\sigma_{2}=1+2\eta H_{\eta}+\eta^{2}G_{\eta},
\]
where $\sigma_{1}$ and $\sigma_{2}$ are singular
values of $P'_{\Gamma}$. 
\end{proof}
This leads to the following proposition: 
\begin{thm}
Consider $\Gamma$ a curve in 2D or surface in 3D with $C^{2}$ boundaries if it is not closed,
and define $d:\mathbb{R}^{n}\mapsto{\mathbb{R}^{+}\cup\{0\}}$ ($n=2,3)$ to
be the distance function to $\Gamma$ with $P_{\Gamma}:\mathbb{R}^{n}\mapsto\Gamma$
the closest point mapping to $\Gamma$. Then for $\epsilon \max_{x \in \Gamma} |\kappa(x)| < 1$ for any $\kappa(x)$ principal curvatures of $\Gamma$ at $x$, we have
\begin{equation}
\int_{\Gamma}v(\mathbf{x)}d\mathbf{x}=\int_{\mathbb{R}^{n}}v(P_{\Gamma}(\mathbf{x})\delta_{\epsilon}(d(\mathbf{x}))\Sigma(\mathbf{x})d\mathbf{x},\label{eq:int_openinterface-1}
\end{equation}
where $\delta_{\epsilon}$ is an averaging kernel 
and $\Sigma(\mathbf{x})$is defined as 
\[
\Sigma(\mathbf{x})=\left\{ \begin{array}{ll}
\sigma_{1}(\mathbf{x}), & n=2,\\
\sigma_{1}(\mathbf{x})\sigma_{2}(\mathbf{x}), & n=3,
\end{array}\right.
\]
where $\sigma_{j}(\mathbf{x})$ , $j=1,2,$ is the $j$-th singular
value of the Jacobian matrix $\displaystyle P'_{\Gamma}$
evaluated at $\mathbf{x}.$ \end{thm}
\begin{proof}
If $\Gamma$ is closed we combine Equation \eqref{intformulation}
with the result $J(\mathbf{x})=\Sigma(\mathbf{x})$ from Equation~\eqref{jacobian_singularvalues}
of Proposition \ref{thm_jacobian_evalues}.

If $\Gamma$ is open there is a little more to show since Equation
\eqref{intformulation} was only derived for closed manifolds. Before
we state the result, it is necessary to understand how $\Gamma_{\eta}$ defined in~\eqref{gamma_eta}  (an $\eta-$level
set of $d$) looks like for an open curve in two dimensions and for a surface with boundaries in three dimensions. 

\noindent
In two dimensions, $\Gamma_{\eta}$ consists of a flat tubular part on either
side of the curve and two semi circles at the two ends of the curve.
See Figure~\ref{fig_levelset_curve2D}.

\noindent 
In three dimensions $\Gamma$ is in general made up of three distinct
parts: the interior part, the edges of the boundary and the corners.
If we assume that $\Gamma$ has $N$ edges then we can write $\Gamma=\Gamma^{o}\cup(\cup_{i=1}^{N}E_{i})\cup(\cup_{i=1}^{N}C_{i})$,
where $\Gamma^{o}$ is the interior of $\Gamma$, $E_{i}$ is the
$i$-th edge of the boundary of $\Gamma$ and $C_{i}$ is its $i$
-th corner. In that setting we can write $\Gamma_{\eta}=I_{\eta}\cup(\cup_{i=1}^{N}T_{i}^{\eta})\cup(\cup_{i=1}^{N}S_{i}^{\eta})$,
where $I_{\eta}$ is the inside portion of $\Gamma_{\eta}$, $T_{i}^{\eta}$
is the cylindrical part of $\Gamma_{\eta}$ representing the set of
points located at a distance $\eta$ from the $i$ -th edge $E_{i}$,
and finally $S_{i}^{\eta}$ is the spherical part of $\Gamma_{\eta}$
representing the set of points located at a distance $\eta$ from
the $i$ -th corner $C_{i}$. See Figure~\ref{fig_levelset_surface3D}.

\noindent
In both cases we need to integrate over $\Gamma_{\eta}$
and then subtract the two semi circles at the two end points of the
curve (in two dimensions) or subtract the portions of sphere at the
corners of the surface and the portions of cylinders at the edges
of the surface (in three dimensions). However, it turns out that the subtraction
is unnecessary since $\Sigma(\mathbf{x})=0$ on each of the subtracted
pieces as shown below.

\paragraph*{Two dimensions.}On the semi-circle around the end point of
a curve, the closest point mapping is {\it constant} since all points on
the semi-circle $\Gamma_{\eta}$ map to the end point. As a result, the singular values of the
Jacobian matrix of the closest point mapping are all zeros and thus
$\Sigma(\mathbf{x})=0$ on the semi-circles around the end points
of a curve. 
\paragraph*{Three dimensions.} As in two dimensions, on the portions of
sphere around a corner point of a surface, the closest point mapping
is {\it constant} and thus $\Sigma(\mathbf{x})=0$. On the portion of cylinders,
the closest point mapping is constant along the radial dimension (one
of the principal directions or singular vector) resulting of the singular
value along that direction to be zero. Since $\Sigma(\mathbf{x})$
is the product of the singular values, it follows that $\Sigma(\mathbf{x})=0$
on the portion of cylinders as well. 
Consequently, Equation~\eqref{eq:int_openinterface-1} holds for any $C^{2}$
curve or surface with $C^2$ boundaries of codimension~1.
\end{proof}

\subsection{Codimension 2}

We consider a $C^{2}$ curve in $\mathbb{R}^{3}$ denoted by $\Gamma$ and let $\gamma(s)$ be a parameterization by arclength of $\Gamma.$
We denote by $d:\mathbb{R}^{3}\mapsto{\mathbb{R}^{+}\cup\{0\}}$ the
distance function to $\Gamma$ and let $P_{\Gamma}:\mathbb{R}^{3}\mapsto\Gamma$
be the closest point mapping to $\Gamma$. We consider a parameterization of the
\emph{tubular part }of the level surface for $\eta\in[0,\epsilon]$
defined as 
\[
\mathbf{x}(s,\theta,\eta)\mathrel{\mathop:}\gamma(s)+\eta\cos\theta\vec{\mathbf{N}}(s)+\eta\sin\theta\vec{\mathbf{B}}(s),
\]
where $\vec{\mathbf{T}}=\frac{d\gamma}{ds}$, $\vec{\mathbf{N}}$
and $\vec{\mathbf{B}}$ constitute the Frenet frame for $\gamma$ as
illustrated in Figure~\ref{fig_frame_diagram}. As in the previous
section, if $\Gamma$ is closed then $d$ is the signed distance function
to $\Gamma$.

If we project a point $\mathbf{x}$ on the tubular part of the level
surface $\Gamma_{\eta}$ defined in~\eqref{gamma_eta}, we have $P_{\Gamma}(\mathbf{x(s,\theta,\eta))=\gamma(s)}$.
If $L$ is the length of the curve it follows that 
\begin{align}
\int_{0}^{2\pi}\int_{0}^{L}g(P_{\Gamma}(\mathbf{x(s,\theta,\eta)))|\mathbf{x_{s}\times\mathbf{x_{\theta}}}|}dsd\theta & =\int_{0}^{2\pi}\int_{0}^{L}g(\gamma(s))\eta(1-\eta\kappa(s)\cos\theta)dsd\theta,\label{eq:integral g}\\
 & =\eta\int_{0}^{L}g(\gamma(s))\int_{0}^{2\pi}(1-\eta\kappa\cos\theta)d\theta ds,\nonumber \\
 & =2\pi\eta\int g(\gamma(s))ds.\nonumber 
\end{align}

Note that the tubular part of the level surface $\Gamma_{\eta}$
does not contain the two hemispheres of $\Gamma_{\eta}$ which are located at the two end points of the curve $\Gamma$. Thus,

\begin{equation}
\int_{\Gamma_{\eta} \setminus\left\{ C_{1}\cup C_{2}\right\} }g(P_{\Gamma}(\mathbf{x}))dS_{\mathbf{x}}=2\pi\eta\int_{\Gamma}gds,\label{eq:intgoverGamma}
\end{equation}
where $C_{1}$ and $C_{2}$ are the two hemispheres of the level surface
$\Gamma_{\eta}$ located at the two end points of the curve
$\Gamma$. Consequently, for sufficiently small $\epsilon$ and by
the coarea formula we obtain

\begin{align*}
\int_{\Gamma}g(\gamma(s))ds & =\frac{1}{2\pi}\int_{0}^{\epsilon}\left(\frac{1}{\eta}\int_{\Gamma_{\eta} \setminus\left\{ C_{1}\cup C_{2}\right\} }g(P_{\Gamma}(\mathbf{x))}\right)K_{\epsilon}(\eta)d\eta,\\
 & =\frac{1}{2\pi}\int_{\mathbb{R}^{3}}g(P_{\Gamma}(\mathbf{x))}\frac{K_{\epsilon}(d)}{d}\chi_{(C_{1}\cup C_{2})^{c}}(\mathbf{x)}d\mathbf{x},
\end{align*}
where $K_{\epsilon}$ is a $C^{1}$ averaging kernel supported in
$[0,\epsilon]$ and $\chi_{(C_{1}\cup C_{2})^{c}}(\mathbf{x)}$ is
the characteristic function of the set $(C_{1}\cup C_{2})^{c}$. 
Because of the term $\frac{K_{\epsilon}(d)}{d}$ in the above equation and for better accuracy, we choose a kernel $K_{\epsilon}$ that satisfies the condition $K_{\epsilon}'(0) =0$.
In
our numerical simulations we consider the kernel
\begin{equation}
K_{\epsilon}^{1,1}(\eta)=\frac{1}{\epsilon}\left(1-\cos\left(2\pi\frac{\eta}{\epsilon}\right)\right)\chi_{[0,\epsilon]}(\eta).\label{eq:kernel_curve3D_1}
\end{equation}

Since the formulation above does not use the two hemispheres located
at both end points of the curve, in order to integrate over the tubular
part of $\Gamma_{\eta}$ only, it is necessary
to subtract the integration over each of the hemispheres $C_{1}$
and $C_{2}$ . The result can be summarized in the following proposition: 
\begin{prop}
Consider a single $C^{2}$ curve $\Gamma$ in $\mathbb{R}^{3}$ parameterized
by $\gamma(s)$ where $s$ is the arclength parameter, and let $d$
be the distance function to $\Gamma$. We define $K_{\epsilon}$ to
be a $C^{1}$ averaging kernel compactly supported in $[0,\epsilon]$
and $P_{\Gamma}:\mathbb{R}^{3}\mapsto\Gamma$ to be the closest point
mapping to $\Gamma$.

If $g$ is a continuous function defined on $\Gamma$ then for sufficiently
small $\epsilon>0$ we have 
\begin{equation}
\int_{\Gamma}g(\gamma(s))ds=\frac{1}{2\pi}\int_{\mathbb{R}^{3}}g(P_{\Gamma}(\mathbf{x)})\frac{K_{\epsilon}(d(\mathbf{x}))}{d(\mathbf{x})}d\mathbf{x}-2\int_{0}^{\epsilon}g(\mathbf{x}_{\eta})\eta K_{\epsilon}(\eta)d\eta,\label{eq:FormulationCurve3D}
\end{equation}
where $\mathbf{x}_{\eta}$ is a point on a sphere of radius $\eta$. 
\end{prop}
Note that for the computation of the length of a curve, the correction
terms given by integrating over both $C_{1}$ and $C_{2}$ is 
\begin{align*}
\int_{0}^{\epsilon}\frac{K_{\epsilon}(\eta)}{\eta}|\mathbb{S}^{1}|d\eta & =\int_{0}^{\epsilon}\frac{K_{\epsilon}(\eta)}{\eta}4\pi\eta^{2}d\eta=2\pi\epsilon. \end{align*}
This simple correction is, however, not suitable for more general
cases that contain multiple curve segments and several integrands.
We shall derive a more elegant and seamless way to perform such correction
in the following section.

Now if we consider a $C^{2}$ curve in three dimensions and let $P_{\Gamma}$
be its closest point mapping, we have the following proposition: 
\begin{thm}
Let $\sigma(\mathbf{x})$ be the nonzero singular value of $P'_{\Gamma}$
and let $g$ be a continuous function defined on $\Gamma$. If $\gamma(s)$
is the arclength parameterization of $\Gamma$ and if $\epsilon \max_{x \in \Gamma }  |\kappa(x)|<1$, where $\kappa (x)$ is the curvature of the curve at $x$, we have 
\begin{equation}
\int_{\Gamma}g(\gamma(s))ds=\frac{1}{2\pi}\int_{\mathbb{R}^{3}}g(P_{\Gamma}(\mathbf{x}))\frac{K_{\epsilon}(d)}{d}\sigma(\mathbf{x})d \mathbf{x},
\label{eq:line_integral_formula}
\end{equation}
where $d$ is the distance function to $\Gamma$.
\end{thm}
\begin{proof}
Since $K_{\epsilon}$ is compactly supported in $[0,\epsilon]$ it is sufficient to consider points in the tubular neighborhood of the curve $\Gamma$. Thus, for $\mathbf{x}$ in the tubular neighborhood, there exists $0 \leq \eta \leq \epsilon$ such that $\mathbf{x}\in\Gamma_{\eta}$.

\paragraph*{Case 1:}

$\mathbf{x}$ is on the \emph{spherical part }of $\Gamma_{\eta}$
corresponding to the $\eta$-distance to either of the two end points
of the curve $\Gamma$. WLOG we assume that $\mathbf{x}$ is at a
distance $\eta$ from the first end point $C_{1}$ parameterized by
$\gamma(0)$. The result is the same if $\mathbf{x}$ is on the other
sphere, i.e. at a distance $\eta$ from the other end point $C_{2}.$
In that case, $P_{\Gamma}(\mathbf{x)}=\gamma(0)$ for all $\mathbf{x}$
on the spherical part so that the Jacobian matrix $P'_{\Gamma}=0$. Therefore, for $\mathbf{x}$ on the spherical part of $\Gamma_{\eta}$,
all singular values of the Jacobian matrix are zero.

\paragraph*{Case 2: }

$\mathbf{x}$ is on the \emph{tubular part }of $\Gamma_{\eta}$. In
that case, if we use the Frenet frame centered at the point $\mathbf{x=\mathbf{x}}(s,\theta,\eta)\in\Gamma_{\eta}$
, we can write $\mathbf{x}$ in the new coordinate system $(\vec{\mathbf{T}},\vec{\mathbf{N}},\vec{\mathbf{B}})$
as

\begin{equation}
\mathbf{x}=\gamma(s)+v\vec{\mathbf{N}}+w\vec{\mathbf{B},}\label{eq:x in tube}
\end{equation}
where $u=0$ is the coordinate of $\mathbf{x}$ along $\vec{\mathbf{T}}$,
$v$ is the coordinate along $\vec{\mathbf{N}}$ and $w$ is the coordinate
along $\vec{\mathbf{B.}}$ Since the projection $P_{\Gamma}(\mathbf{x})=\gamma(s)$
does not depend on $v$ nor $w$ (since the plane $(\vec{\mathbf{N}},\vec{\mathbf{B}})$
is normal to the curve $\Gamma$) it follows that

\[
\frac{\partial P_{\Gamma}(\mathbf{x})}{\partial v}=\frac{\partial P_{\Gamma}(\mathbf{x})}{\partial w}=0.
\]
On the other hand, we have 
\[
\frac{\partial P_{\Gamma}(\mathbf{x})}{\partial u}=\frac{\partial\gamma(s)}{\partial u}=\frac{\partial s}{\partial u}\frac{\partial\gamma(s)}{\partial s}=\frac{\partial s}{\partial u}\vec{\mathbf{T},}
\]
where $\frac{\partial s}{\partial u}$ is the variation of the arclength
parameter $s$ with respect to $u$ when the point $\mathbf{x}$ is
moving on $\Gamma_{\eta}$ along the tangential direction $\vec{\mathbf{T}}$.
Since $u$ is the arclength parameter along the tangential direction $\vec{\mathbf{T}}$, it follows
that we have a unit speed parameterization along $\vec{\mathbf{T}}$
giving the identity 
\[
\frac{\partial\mathbf{x}}{\partial u}\cdot\vec{\mathbf{T}}=1.
\]
In addition, 
\begin{align*}
\frac{\partial\mathbf{x}}{\partial s} & =\frac{\partial\gamma(s)}{\partial s}+v\frac{\vec{\mathbf{N}}}{\partial s}+w\frac{\vec{\mathbf{B}}}{\partial s}\\
 & =\vec{\mathbf{T}}-\kappa v\vec{\mathbf{T}}+\tau v\vec{\mathbf{B}}-\tau w\vec{\mathbf{N}}\\
 & =\left(1-\kappa v\right)\vec{\mathbf{T}}-\tau w\vec{\mathbf{N}}+\tau v\vec{\mathbf{B}},
\end{align*}
where $\kappa$ is the curvature of $\Gamma$ at $\gamma(s)$ and
$\tau$ is the torsion of the curve $\Gamma$ at the point $\gamma(s)$.
Since the level surface $\Gamma_{\eta}$ is a tube of radius $\eta$,
its intersection with the normal plane $(\vec{\mathbf{N}},\vec{\mathbf{B}})$
is a circle of radius $\eta$. Hence if we use polar coordinates on the normal plane, we obtain $v=\eta\cos\theta$ and $w=\eta\sin\theta$
. It follows that 
\[
\frac{\partial\mathbf{x}}{\partial s}\cdot\vec{\mathbf{T}}=1-\kappa\eta\cos\theta.
\]
Consequently we have 
\[
\frac{\partial\mathbf{x}}{\partial u}\cdot\vec{\mathbf{T}}=1=\frac{\partial s}{\partial u}\frac{\partial\mathbf{x}}{\partial s}\cdot\vec{\mathbf{T}}=\frac{\partial s}{\partial u}(1-\kappa\eta\cos\theta),
\]
and 
\[
\frac{\partial s}{\partial u}=\frac{1}{1-\kappa\eta\cos\theta}.
\]
Therefore, in the Frenet frame, the Jacobian matrix of the closest
point projection map can be written as 
\[
P'_{\Gamma}=\left(\begin{array}{ccc}
\frac{1}{1-\kappa\eta\cos\theta} & 0 & 0\\
0 & 0 & 0\\
0 & 0 & 0
\end{array}\right),
\]
where $\frac{1}{1-\kappa\eta\cos\theta}$ is the nonzero eigenvalue
of the Jacobian of the closest point mapping. Based on the hypothesis on the size of $\epsilon$ related to the geometry of the curve $\Gamma$, the term $\frac{1}{1-\kappa\eta\cos\theta}$ is strictly positive and therefore is also the singular value $\sigma(\mathbf{x})$ of the Jacobian of the closest point mapping.

Therefore we have 
\begin{equation}
\sigma(\mathbf{x)=\begin{cases}
0 & \mbox{if }\mathbf{x}\mbox{ is on the spherical part of }\Gamma_{\eta},\\
\frac{1}{1-\kappa\eta\cos\theta} & \mbox{if }\mathbf{x}\mbox{ is on the tubular part of }\Gamma_{\eta}.
\end{cases}}\label{eq:sigma}
\end{equation}
Now using \eqref{eq:integral g} and \eqref{eq:intgoverGamma} we
obtain 
\[
\begin{aligned}\int_{\Gamma_{\eta}}g(P_{\Gamma}(\mathbf{x))}\sigma(\mathbf{x)}dS_{\mathbf{x}} & =\int_{\Gamma_{\eta} \setminus\left\{ C_{1}\bigcup C_{2}\right\} }g(P_{\Gamma}(\mathbf{x))}\sigma(\mathbf{x)}dS_{\mathbf{x}}\\
 & =\int_{0}^{2\pi}\int_{0}^{L}g(P_{\Gamma}(\mathbf{x))}\sigma(\mathbf{x)}|\mathbf{\mathbf{x_{s}\times\mathbf{x_{\theta}}}}|dsd\theta\\
 & = \int_{0}^{2\pi}\int_{0}^{L}g(\gamma(s))\eta\frac{1-\eta\kappa(s)\cos\theta}{1-\eta\kappa(s)\cos\theta}dsd\theta\\
 & =2\pi\eta\int_{0}^{L}g(\gamma(s))ds 
\end{aligned}
\]
It follows that for $K_{\epsilon}$ a $C^{1}$ averaging kernel compactly
supported in $[0,\epsilon]$, for sufficiently small $\epsilon$ and
by the coarea formula, we have

\[
\begin{aligned}\int_{\Gamma}gds & =\frac{1}{2\pi}\int_{0}^{\epsilon}\frac{1}{\eta}\int_{\Gamma_{\eta}}g(P_{\Gamma}(\mathbf{x))\sigma(\mathbf{x)}}K_{\epsilon}(\eta)d\eta\\
 & =\frac{1}{2\pi}\int_{\mathbb{R}^{3}}g(P_{\Gamma}(\mathbf{x))}\frac{K_{\epsilon}(d)}{d}\sigma(\mathbf{x)}d\mathbf{x.}
\end{aligned}
\]

\end{proof}

\section{Numerical simulations\label{sec:Numerical-simulations}}

In this section we investigate the convergence of our numerical integration using simple 
Riemann sums over uniform Cartesian grids. Unless stated otherwise, the singular values are computed from the matrix the elements of which are computed by
 the standard central difference approximations of the Jacobian matrix $P'_{\Gamma}.$ In other words, the Jacobian matrix $P'_{\Gamma}$ is computed by using finite differences to evaluate the partial derivatives of each component of $P_\Gamma (\mathbf{x})$; more precisely,  if $P_\Gamma(\mathbf{x}) ~= ~(p_{1}(\mathbf{x}), p_{2}(\mathbf{x}), p_{3}(\mathbf{x}))$, and $\mathbf{x}~=~(x_{1},x_{2},x_{3})$  we use finite difference to approximate $\displaystyle \frac{\partial p_{j}}{\partial x_{k}}$ for $1 \leq j,k \leq 3$. We do not evaluate the expressions that involve the partial derivatives of the distance function.

In our computations
we use the cosine kernel
\begin{equation} \label{cosine_kernel}
K_{\epsilon}^{\mbox{cos}}(\eta)=\chi_{[-\epsilon,\epsilon]}(\eta)\frac{1}{2\epsilon}\left(1+\cos\left(\frac{\pi\eta}{\epsilon}\right)\right)
\end{equation}
for integration on surfaces of  codimension 1, and the kernel $K_{\epsilon}^{1,1}$ defined in \eqref{eq:kernel_curve3D_1} 
for codimension 2. 
With these compactly supported kernels, formulas \eqref{eq:int_openinterface-1} and 
\eqref{eq:line_integral_formula} can be considered integration of 
functions defined on suitable hypercubes, periodically extended.
In such settings, simple Riemann sums on Cartesian grids are equivalent to sums using 
Trapezoidal rule, and if all the terms are known analytically, the order of accuracy will be related in general to the smoothness of the
integrands; exception can be found when the normals of the surfaces are rationally dependent on the 
step sizes used in the Cartesian grids. 

\subsection{Integration of codimension one surfaces}
We tested our numerical integration on two different
portions of circle, a torus, a quarter sphere and a three quarter
sphere. We computed their respective lengths
or surface areas by integrating the constant $1$ over the curve or surface.
Each of these tests were designed to exhibit the convergence rate of
our formulations on cases with varying difficulty. In particular,
the convergence rate of our formulation depends on the smoothness
of the closest point mapping inside the tubular neighborhood of the curve or surface.

The results for the portions of circle are given in  Tables~\ref{table_conv_portioncircle_easy}
and \ref{table_conv_portioncircle_difficult}. In the first convergence
studies (Table \ref{table_conv_portioncircle_easy}) the line where
the closest point mapping has a jump discontinuity is parallel to
the grid lines. In this case we see a second order convergence rate
using central differencing to compute the Jacobian
matrix $P'_{\Gamma}$. In the second test
case however, the portion of circle is chosen so that the line where
the closest point mapping has a jump discontinuity is not parallel
to the grid lines. In that case the normal to the curve is rationally dependent on the step size of the Cartesian grid and the convergence rate reduces to first order
even though we used central differencing to compute $P'_{\Gamma}$. We note that in these two tests,
we chose $\epsilon$ (the half width of the tubular neighborhood around
the curve) small enough so that the line where the closest point mapping
is discontinuous is outside of it.

In three dimensions we first tested our method on a torus (closed smooth
surface). The results for the torus are reported in Table \ref{table_conv_torus}.  In this case the closest point mapping is very smooth and we see
third order convergence when using the exact signed distance function and a third order difference scheme
to approximate $P'_{\Gamma}$ (see $\mbox{RE}_{\infty}$ in Table~\ref{table_conv_torus}).
We also tested our method with a computed signed distance function. We constructed the signed distance function using the algorithm described in \cite{cheng_tsai08}, and compared the performance of our method with a fourth order accurate signed distance function and a first order accurate signed distance function (see $\mbox{RE}_{4}$ and  $\mbox{RE}_{1}$ in Table~\ref{table_conv_torus}.) With the fourth order accurate signed distance function we used a third order accurate difference scheme to approximate  $P'_{\Gamma}$, and with the first order accurate signed distance function we used a second order accurate difference scheme to approximate  $P'_{\Gamma}$.  

For surfaces with boundaries we tested the method on a quarter sphere and a
three quarter sphere. The three quarter sphere case is illustrated in Figure~\ref{fig_3quartersphere}. 
The reason for choosing these two cases is because
the closest point mapping has a different degree of smoothness for
each of these surfaces. For the quarter sphere the closest point
mapping is smooth enough, but for the three quarter sphere, the tubular
neighborhood around the surface contains the line where the closest
point mapping has a jump discontinuity. In that latter case, it is
therefore necessary to use an adequate one sided discretization to
compute $P'_{\Gamma}$ accurately. 
The one-sided discretization that we used is reported in Section~\ref{sec:one-sided-discretization}.
The test for the quarter sphere still uses central differencing to compute
$P'_{\Gamma}$. The results for the portions
of sphere are reported in Tables~\ref{table_conv_quartersphere} and
\ref{table_conv_3quartersphere}.



\subsection{Integrating along curves in three dimensions}
In codimension 2, we tested our numerical integration on a coil wrapped
around the helix defined parametrically as  
\[
\mathbf{x}(t)=\left(
r\cos(t),
r\sin (t),
bt
\right),
\]
with $r=0.75$ and $b=0.25$. The coil is then wrapped around the
helix at a distance of $0.2$ from the helix. See Figure~\ref{fig:coil}.
As our test case, we computed the length of the coil by integrating
$1$ along the curve. The results are reported in Table~\ref{table_conv_helix}.

\subsection{One-sided discretization of the Jacobian matrix}\label{sec:one-sided-discretization}
Here  for completeness, we describe the one-sided discretization used in computing results reported in Table~\ref{table_conv_3quartersphere}. For simplicity we provide the explanation in $\mathbb R^2$. The discretization generalizes easily to 3D.

We will describe the one-sided discretization for a
uniform Cartesian grid in $\mathbb R^2$, namely for $P_{\Gamma}(\mathbf{x}_{i,j})=(U_{i,j},V_{i,j})$
with $\mathbf{x}_{i,j}=(ih,jh),$ $i,j\in\mathbb{Z}$ and $h>0$ being
the step size. The Jacobian matrix will be approximated by simple
finite differences defined below: 
\[
P'_{\Gamma}(\mathbf{x}_{i,j})\approx\left(\begin{array}{cc}
(U_{x})_{i,j} & (U_{y})_{i,j}\\
(V_{x})_{i,j} & (V_{y})_{i,j}
\end{array}\right).
\]
The discretization of $U$ and $V$ have to be defined together because
the two functions are not independent of each other. With

\[
(U_{x}^{\pm})_{i,j}:=\pm\frac{1}{2h}\left(-3U_{i,j}+4U_{i\pm1,j}-U_{i\pm2,j}\right),
\]
 and the smoothness indicator
\[
S^\pm_{i,j}=S^{\pm}(U_{i,j}):=\triangle^{+}\triangle^{-}U_{i\pm1,j}
\]
 we define 
\[
(U_{x})_{i,j}:=\begin{cases}
(U_{x}^{+})_{i,j}, & \mbox{if }|S_{i,j}^{+}|\le|S_{i,j}^{-}|,\\
(U_{x}^{-})_{i,j}, & \mbox{otherwise, }
\end{cases}
\]
and $(V_{x})_{i,j}$ is defined according to the choice of stencil
based on $S^{\pm}(U_{i,j})$ 
\[
(V_{x})_{i,j}:=\begin{cases}
(V_{x}^{+})_{i,j}, & \mbox{if }|S_{i,j}^{+}|\le|S_{i,j}^{-}|,\\
(V_{x}^{-})_{i,j}, & \mbox{otherwise.}
\end{cases}
\]
The discretization of $U_{y}$ and $V_{y}$ is defined similarly with the choice of the stencil
based on $S^{\pm}(V_{i,j})$. 

\section{Summary\label{sec:Conclusion}}

In this paper, we presented a new approach for computing integrals
along curves and surfaces that are defined either implicitly by the
distance function to these manifolds or by the closest point mappings.
We are motivated by the abundance of discrete point sets sampled from
surfaces using devices such as LIDAR, the need to compute functionals
defined over the underlying surfaces, as well as many applications
involving the level set method or the use of closest point methods.

Contrary to most other existing approximations using either smeared
out Dirac delta functions or locally obtained parameterized patches,
we derive a volume integral in the embedding Euclidean space which
is equivalent to the desired surface or line integrals. This allows
for easy construction of higher order numerical approximations of
these integrals. The key components of this new approach include the
use of singular values of the Jacobian matrix of the closest point
mapping, which can be computed easily to high order even by simple
finite differences.


\begin{backmatter}

\section*{Competing interests}
  The authors declare that they have no competing interests.


\section*{Acknowledgements}

The second author thanks Prof. Steve Ruuth for stimulating conversations. Kublik's research was partially funded by a University of Dayton Research Council Seed Grant and Tsai's research is partially supported by Simons Foundation, NSF Grants
DMS-1318975, DMS-1217203, and ARO Grant No. W911NF-12-1-0519.



\bibliographystyle{bmc-mathphys} 
\bibliography{references}      




\section*{Figures}
  \begin{figure}[h!]
  \caption{\csentence{Level set of a 2D open curve.}
     An example of an open curve $\Gamma$ (black curve) and its $\eta$-level
set $\Gamma_{\eta}$ (red curve). $\Gamma_{\eta}$ consists of a tubular
part and two semi circles at the two ends.} \label{fig_levelset_curve2D}
      \end{figure}

\begin{figure}[h!]
  \caption{\csentence{Level set of a 3D surface with boundaries.}
      An example of a surface with boundaries viewed from different angles
and its corresponding $\eta$-level set $\Gamma_{\eta}$ viewed from
the same angles. The figure at the bottom right corner
shows the surface and $\Gamma_{\eta}$.} \label{fig_levelset_surface3D} 
      \end{figure}

  \begin{figure}[h!]
  \caption{\csentence{Level set of an open curve in 3D.}
     Three dimensional curve with its $\eta$-level surface $\Gamma_{\eta}$ in green and the Frenet
frame at a point on $\Gamma_{\eta}$. }  \label{fig_frame_diagram}
      \end{figure}

 \begin{figure}[h!]
  \caption{\csentence{Three quarter sphere.}
    The three quarter sphere and its corresponding $\eta$-level set $\Gamma_{\eta}$.}   \label{fig_3quartersphere} 
      \end{figure}
      
       \begin{figure}[h!]
  \caption{\csentence{Coil and one of its level sets.}
    The coil and one of the level sets of the distance function to the coil used in the reported numerical simulations.} \label{fig:coil} 
      \end{figure}


\section*{Tables}


\begin{table}[h!]
\caption{Errors for a portion of circle. Relative errors in the numerical approximation of the length of a planar curve, which is
a portion of circle of radius $R=0.75$ centered at $0$. 
The width for the tubular neighborhood of the curve is $\epsilon=0.2.$ 
In this computation,  the closest point mapping has a jump discontinuity along a straight-line 
which is arranged to be parallel to the grid lines. 
\label{table_conv_portioncircle_easy} }
 %
\begin{tabular}{|c|c|c|}
\hline 
$n$  & Relative Error  & Order\tabularnewline
\hline 
\hline 
$64$  & $2.7994\times10^{-4}$  & --\tabularnewline
\hline 
$128$  & $7.0665\times10^{-5}$  & $1.99$\tabularnewline
\hline 
$256$  & $1.7187\times10^{-5}$  & $2.04$\tabularnewline
\hline 
$512$  & $4.2719\times10^{-6}$  & $2.01$\tabularnewline
\hline 
$1024$  & $1.0636\times10^{-6}$  & $2.01$\tabularnewline
\hline 
$2048$  & $2.6567\times10^{-7}$  & $2.00$\tabularnewline
\hline 
$4096$  & $6.6045\times10^{-8}$  & $2.01$\tabularnewline
\hline 
$8192$  & $1.6513\times10^{-8}$  & $2.00$\tabularnewline
\hline 
\end{tabular}
\end{table}

\begin{table}[h!]
\caption{Errors for a tilted portion of circle. Relative errors in the numerical approximation of the length of a planar curve, 
which is a portion of circle of radius $R=0.75$ centered at $0$. 
The width for the tubular neighborhood of the curve is $\epsilon=0.2.$ 
In this computation,  the  jump discontinuity of the closest point mapping is not parallel to the grid lines. \label{table_conv_portioncircle_difficult} }
 %
\begin{tabular}{|c|c|c|}
\hline 
$n$  & Relative Error  & Order\tabularnewline
\hline 
\hline 
$64$  & $3.7159\times10^{-5}$  & --\tabularnewline
\hline 
$128$  & $2.5786\times10^{-7}$  & $7.17$\tabularnewline
\hline 
$256$  & $4.2361\times10^{-6}$  & $-4.04$\tabularnewline
\hline 
$512$  & $3.2246\times10^{-6}$  & $0.39$\tabularnewline
\hline 
$1024$  & $1.8876\times10^{-6}$  & $0.77$\tabularnewline
\hline 
$2048$  & $1.0132\times10^{-7}$  & $0.90$\tabularnewline
\hline 
$4096$  & $5.2372\times10^{-7}$  & $0.95$\tabularnewline
\hline 
$8192$  & $2.6615\times10^{-7}$  & $0.98$\tabularnewline
\hline 
\end{tabular}
\end{table}

\begin{table}[h!]
\caption{Errors for a torus. Relative errors in the numerical approximation of the surface area of a torus centered at $0$. The distance from the center to the tube that form the torus is $R=0.75$ and the radius of the tube
is $r=0.25$. 
In this computation, we summed up grid points that are within  $\epsilon=0.2$ distance from the surface for $ \mbox{RE}_{\infty}$ and $\mbox{RE}_{4}$, and $\epsilon = 0.03$ for $\mbox{RE}_{1}$. $ \mbox{RE}_{\infty}$, $\mbox{RE}_{4}$ and $\mbox{RE}_{1}$ are the relative error using the exact signed distance function, the relative error using a fourth order accurate signed distance function and the relative error using a first order accurate signed distance function respectively.
The Jacobian matrix $P'_{\Gamma}$  is approximated by  a 
standard third order accurate differencing except for  $\mbox{RE}_{1}$  where we used a second order accurate differencing to approximate $P'_{\Gamma}$. \label{table_conv_torus} }
 %
\begin{tabular}{|c|c|c|c|c|c|c|}
\hline 
$n$  & $ \mbox{RE}_{\infty}$  & Order & $\mbox{RE}_{4}$ & Order & $\mbox{RE}_{1}$ & Order \tabularnewline
\hline 
\hline 
$32$  & $6.2030\times10^{-3}$  & $-$ & $1.1699\times10^{-2}$  & $-$ & $5.8000\times10^{-2}$  & $-$ \tabularnewline
\hline 
$64$  & $1.8073\times10^{-4}$  & $5.10$& $1.0169\times10^{-3}$  & $3.52$& $1.4456\times10^{-2}$  & $2.00$  \tabularnewline
\hline 
$128$  & $6.6838\times10^{-6}$  & $4.76$& $1.3568\times10^{-5}$  & $6.23$ & $3.9830\times10^{-3}$  & $1.86$  \tabularnewline
\hline 
$256$  & $4.1530\times10^{-7}$  & $4.01$& $7.1567\times10^{-7}$  & $4.24$ & $1.4391\times10^{-3}$  & $1.47$ \tabularnewline
\hline 
$512$  & $5.0379\times10^{-8}$  & $3.04$ & $6.1982\times10^{-8}$& $3.53$ & $5.1463 \times 10^{-4}$ & 1.48   \tabularnewline
\hline 
\end{tabular}
\end{table}

\begin{table}[h!]
\caption{Errors for a quarter sphere. 
Relative errors in the numerical approximation of the surface area of a 
quarter sphere with radius $R=0.75$ centered at $0.$
In this computation, we summed up grid points that are within  $\epsilon=0.2$ distance from the surface.
We used the standard central
difference scheme to compute each entry of the Jacobian matrix $P'_{\Gamma}$.
\label{table_conv_quartersphere} }
 %
\begin{tabular}{|c|c|c|}
\hline 
$n$  & Relative Error  & Order\tabularnewline
\hline 
\hline 
$32$  & $9.2825\times10^{-3}$  & $-$\tabularnewline
\hline 
$64$  & $1.8365\times10^{-3}$  & $2.34$\tabularnewline
\hline 
$128$  & $2.7726\times10^{-4}$  & $2.73$\tabularnewline
\hline 
$256$  & $7.1886\times10^{-5}$  & $1.95$\tabularnewline
\hline 
$512$  & $1.4811\times10^{-5}$  & $2.30$\tabularnewline
\hline 
\end{tabular}
\end{table}

\begin{table}[h!]
\caption{Errors for a three quarter sphere. Relative errors in the numerical approximation of the surface area of a 
three quarter sphere with radius $R=0.75$ centered at $0$ (this is the portion of a sphere that misses half of a hemisphere).
In this computation, we summed up grid points that are within  $\epsilon=0.2$ distance from the surface.
Due to this setup, the closest point mapping has a discontinuity that stems out from the boundary of the surface.
We used the discretization described in Section~\ref{sec:one-sided-discretization} to compute each entry of the Jacobian
matrix $P'_{\Gamma}$. \label{table_conv_3quartersphere} }
 %
\begin{tabular}{|c|c|c|}
\hline 
$n$  & Relative Error  & Order\tabularnewline
\hline 
\hline 
$32$  & $1.1726\times10^{-2}$  & $-$\tabularnewline
\hline 
$64$  & $1.1733\times10^{-3}$  & $3.32$\tabularnewline
\hline 
$128$  & $9.1325\times10^{-4}$  & $0.36$\tabularnewline
\hline 
$256$  & $3.8238\times10^{-4}$  & $1.26$\tabularnewline
\hline 
$512$  & $7.8308\times10^{-5}$  & $2.29$\tabularnewline
\hline 
\end{tabular}
\end{table}

\begin{table}[h!]
\caption{Errors for a coil. Relative errors in the numerical approximation of 
a coil wrapped around a helix. In this computation, we used a constant width for
the tubular neighborhood $\epsilon=0.1$ and took the averaging kernels
to be $K_{\epsilon}^{1,1}$ defined in~\eqref{eq:kernel_curve3D_1}. \label{table_conv_helix} }
 %
\begin{tabular}{|c|c|c|}
\hline 
$n$  & Relative Error  & Order \tabularnewline
\hline 
\hline 
$60$  & $5.5078\times10^{-3}$  & $-$ \tabularnewline
\hline 
$120$  & $1.1476\times10^{-3}$  & $2.63$ \tabularnewline
\hline 
$240$  & $2.3409\times10^{-4}$  & $2.29$ \tabularnewline
\hline 
$480$  & $3.7166\times10^{-5}$  & $2.66$ \tabularnewline
\hline 
\end{tabular}
\end{table}

\newpage


%
%

\end{backmatter}
\end{document}